\documentclass[a4paper,12pt]{article}
\usepackage{kotex}
\usepackage{amsmath,amssymb,amsthm, mathrsfs}
\usepackage{latexsym}
\usepackage{hyperref}
\usepackage[bf, small]{titlesec}
\usepackage{authblk} %%% new line between authors

\usepackage{lineno}
\theoremstyle{plain}
\newtheorem{Thm}{Theorem}[section]
\newtheorem{Lem}[Thm]{Lemma}
\newtheorem{Prop}[Thm]{Proposition}
\newtheorem{Cor}[Thm]{Corollary}

\theoremstyle{definition}

\usepackage{standalone}
\usepackage{tikz}
\usetikzlibrary{arrows,positioning,matrix,fit,backgrounds,shapes,shapes.geometric, intersections}

\usetikzlibrary{shapes.callouts,decorations.pathmorphing} %%% 팝업용
\usetikzlibrary{shadows} % shadow
\usepackage{calc} % calculate angles to evenly space vertices in circular arrangements.
\usetikzlibrary{decorations.markings} %%put arrowheads in the middle of directed edges
\tikzstyle{vertex}=[circle, draw, inner sep=0pt, minimum size=6pt] % style

\usetikzlibrary{arrows,matrix} %%% Hesse Diagram

\newcommand{\RRR}{\mathcal{R}} %%%%% mathcal
\title{On competition indices and periods of multipartite tournaments}
\author[1]{\small Ji-Hwan Jung}
\author[2]{\small Suh-Ryung Kim}
\author[2]{\small Hyesun Yoon}
\affil[1]{\footnotesize Center for Educational Research, Seoul National University, Seoul 08826}
\affil[2]{\footnotesize Department of Mathematics Education, Seoul National University, Seoul 08826}
\affil[ ]{\footnotesize\textit{jihwanjung@snu.ac.kr, srkim@snu.ac.kr, magisakura@snu.ac.kr}}
\date{}
\begin{document}
\maketitle
\begin{abstract}
In this paper, we compute competition indices and periods of multipartite tournaments.
We first show that the competition period of an acyclic digraph $D$ is one and $\zeta(D) +1$ is a sharp upper bound of the competition index of $D$ where $\zeta(D)$ is the sink elimination index of $D$.
Then we prove that, especially, for an acyclic $k$-partite tournament $D$, the competition index of $D$ is $\zeta(D)$ or $\zeta(D) +1$ for an integer $k \ge 3$.
By developing useful tools to create infinitely many directed walks in a certain regular pattern from given directed walks, we show that the competition period of a multipartite tournament with sinks and directed cycles is at most three.
We also prove that the competition index of a primitive digraph does not exceed its exponent.
\end{abstract}


\section{Introduction}
The underlying graph of each digraph in this paper is assumed to be simple unless otherwise mentioned.
%Throughout this paper, we assume that the underlying graph of each digraph is simple unless otherwise mentioned.

Given a digraph $D$, the {\em competition graph} $C(D)$ of $D$
has the same vertex set as $D$ and has an edge between vertices $u$ and $v$
if and only if there exists a common prey of $u$ and $v$ in $D$. If $(u,v)$  is an arc of a digraph $D$,
then we call $v$ a {\em prey} of $u$ (in $D$) and call $u$ a {\em predator} of $v$ (in $D$).
The notion of competition graph is due to Cohen~\cite{cohen1968interval} and has arisen from ecology. Competition graphs also have applications in coding, radio transmission, and modeling of complex economic systems.  (See \cite{raychaudhuri1985generalized} and \cite{roberts1999competition} for a summary of these applications.)
Various variants of notion of competition graphs have been introduced and studied (see the survey articles by Kim~\cite{kim1993competition} and Lundgren~\cite{lundgren1989food} for the variations which have been defined and studied by many authors since Cohen introduced the notion of competition graph).

The notion of $m$-step competition graph is one of the important variants and is defined as follows.
Given a digraph $D$ and a positive integer $m$, a vertex $y$ is an {\em $m$-step prey} of a vertex $x$ if and only if there exists a directed walk from $x$ to $y$ of length $m$. Given a digraph $D$ and a positive integer $m$, the digraph $D^m$ has the vertex set same as $D$ and has an arc $(u,v)$ if and only if $v$ is an $m$-step prey of $u$.
Given a positive integer $m$, the {\em $m$-step competition graph} of a digraph $D$, denoted by $C^m(D)$, has the same vertex set as $D$ and has an edge between vertices $u$ and $v$ if and only if there exists an $m$-step common prey of $u$ and $v$. The notion of $m$-step competition graph is introduced by Cho~{\em et al.}~\cite{cho2000m} as a generalization of competition graph.
By definition, it is obvious that $C^1(D)$ for a digraph $D$ is the competition graph $C(D)$.
Since its introduction, it has been extensively studied (see for example \cite{belmont2011complete,cho2011competition,helleloid2005connected,ho2005m,kim2008competition,park2011m,zhao2009note}). Cho~{\em et al.}~\cite{cho2000m} showed that for any digraph $D$ and a positive integer $m$, $C^m(D)=C(D^m)$.

For the two-element Boolean algebra $\mathcal{B}=\{0,1\}$, $\mathcal{B}_n$ denotes the set of all $n \times n$ (Boolean) matrices over $\mathcal{B}$. Under the Boolean operations, we can define matrix
addition and multiplication in $\mathcal{B}_n$.
A graph $G$ is called the {\em row graph} of a matrix $A \in \mathcal{B}_n$ and denoted by $\RRR(A)$ if the rows of $A$ are the vertices of $G$, and two vertices are adjacent in $G$ if and only if their corresponding rows have a nonzero entry in the same column of $A$.
This notion was studied by Greenberg~{\em et al.}~\cite{greenberg1984inverting}. As noted in \cite{greenberg1984inverting}, the competition graph of a digraph $D$ is the row graph of its adjacency matrix.

Cho and Kim~\cite{cho2004competition} introduced the notions of competition index and competition period of $D$ for a strongly connected digraph $D$,
and Kim~\cite{kim2008competition} extended these notions to a general digraph $D$.
Consider the graph sequence $C^1(D)$, $C^2(D)$, $C^3(D), \ldots, C^m(D)$, $\ldots$ for a digraph $D$.
Note that for a digraph $D$ and its adjacency matrix $A$, the graph sequence $C^1(D)$, $C^2(D), \ldots, C^m(D)$, $\ldots$ is equivalent to the row graph sequence $\mathcal{R}(A)$, $\mathcal{R}(A^2), \ldots, \mathcal{R}(A^m), \ldots$.
Since the cardinality of the Boolean matrix set $\mathcal{B}_n$ is equal to a finite number $2^{n^2}$, there is a smallest positive integer $q$ such that $C^{q+i}(D)=C^{q+r+i}(D)$ equivalently $\mathcal{R}(A^{q+i})=\mathcal{R}(A^{q+r+i})$ for some positive integer $r$ and all nonnegative integers $i$.
Such an integer $q$ is called the \emph{competition index} of $D$ and is denoted by cindex$(D)$.
For $q=$cindex$(D)$, there is also a smallest positive integer $p$ such that $C^{q}(D)=C^{q+p}(D)$ equivalently $\mathcal{R}(A^{q})=\mathcal{R}(A^{q+p})$.
Such an integer $p$ is called the \emph{competition period} of $D$ and is denoted by cperiod$(D)$.
Refer to \cite{kim2010generalized, kim2015characterization, kim2012bound} for some results of competition indices and competition periods of digraphs.

Eoh {\it et al.}~\cite{eoh2020m} studied the $m$-step competition graphs of orientations of complete bipartite graphs for an integer $m \ge 2$.
They introduced a notion of sink sequences of digraphs, which played a key role in the paper.
Given a digraph $D$, we call a vertex of outdegree zero a \emph{sink} in $D$. We define a nonnegative integer $\zeta(D)$ and sequences
\[(W_0, W_1, \ldots, W_{\zeta(D)}) \quad \mbox{and} \quad (D_0, D_1, \ldots, D_{\zeta(D)}) \]
 of subsets of $V(D)$ and subdigraphs of $D$, respectively, as follows.
Let $D_0=D$ and $W_0$ be the set of sinks in $D$.
If $W_0 = V(D)$ or $W_0 = \emptyset$, then let $\zeta(D)=0$.
Otherwise, let $D_1 = D_0-W_0$ and let $W_1$ be the set of sinks in $D_1$.
If $W_1 = V(D_1)$ or $W_1 = \emptyset$, then let $\zeta(D)=1$.
Otherwise, let $D_2 = D_1-W_1$ and let $W_2$ be the set of sinks in $D_2$.
If $W_2 = V(D_2)$ or $W_2 = \emptyset$, then let $\zeta(D)=2$.
We continue in this way until we obtain $W_k=V(D_k)$ or $W_k =\emptyset$ for some nonnegative integer $k$.
Then we let $\zeta(D)=k$.
By definition, $0 \le \zeta(D) \le |V(D)|-1$.
We call $\zeta(D)$ the \emph{sink elimination index} of $D$, the sequence $(W_0, W_1, \ldots, W_{\zeta(D)})$ the \emph{sink sequence} of $D$, and the sequence $(D_0, D_1, \ldots, D_{\zeta(D)})$ the \emph{digraph sequence associated with the sink sequence} of $D$.

In this paper, we study competition indices and competition periods of multipartite tournaments in terms of sink sequences of digraphs.
A \emph{$k$-partite tournament} is an orientation of a complete $k$-partite graph for a positive integer $k$.
In particular, if $k \ge 2$, then we call it a \emph{multipartite tournament}.

Two vertices $u$ and $v$ in a digraph $D$ are
said to be {\em strongly connected} if there are directed walks from
$u$ to $v$ and from $v$ to $u$. We say that a digraph $D$ is {\em
strongly connected} if each pair in $V(D)$ is strongly connected. A
digraph $D$ is said to be {\em primitive} if $D$ is strongly
connected and the greatest common divisor of lengths of its directed
cycles is equal to $1$. It is known \cite{brualdi1991combinatorial} that
if a digraph $D$ is primitive, then there exists a positive integer
$t$ such that there is a directed walk of length exactly $t$ from each vertex
$u$ to each vertex $v$ (possibly $u=v$). The smallest integer $t$ is called the {\it exponent of the
primitive digraph $D$} and it is denoted by $\exp(D)$.

In Section~\ref{sec:acyclic}, we deal with acyclic multipartite tournaments.
We show that the competition period of an acyclic digraph $D$ is one and $\zeta(D) +1$ is a sharp upper bound of the competition index of $D$ (Theorem~\ref{acyclic-digraph-properties}).
Especially, it turns out that the competition index of an acyclic $k$-partite tournament $D$ is $\zeta(D)$ or $\zeta(D) +1$ for an integer $k \ge 3$ (Theorem~\ref{Thm:acyclic k-partite tournament}).
In Section~\ref{sec:directed cycle}, we handle multipartite tournaments with sinks and directed cycles.
We introduce types of a directed walk and types of a vertex in $\bigcup_{i=0}^{\zeta(D)-1}W_i$ where $\left(W_0,\ldots ,W_{\zeta(D) }\right)$ is the sink sequence of a multipartite tournament $D$ with sinks and directed cycles, and then show that each vertex in $\bigcup_{i=0}^{\zeta(D)-1}W_i$ is of Type~1 or Type~2 (Theorem~\ref{thm:34cycle}).
We show that the existence of $(u,w)$-directed walk of Type~1 or Type~2 of a certain length for a vertex $u$ in $D_{\zeta(D)}$ and $w \in \bigcup_{i=0}^{\zeta(D)-1}W_i$ guarantees the existence of a $(u,w)$-directed walk of length $m$ in $D$ for infinitely many integers $m$ in a specific form (Lemmas~\ref{lem:3cycle}-\ref{lem:both34}).
By integrating these results, we show that the competition period of a multipartite tournament with sinks and directed cycles is at most three (Theorem~\ref{thm:at most 3}).
In Section~\ref{sec:primitive}, we show that the competition index of a primitive digraph is at most its exponent (Theorem~\ref{thm:primitive digraph}).
In Section~\ref{sec:tournaments}, we take care of it to eventually compute the competition period and competition index of a tournament with sinks (Theorem~\ref{thm:tournament2}).

For a positive integer $n$, we denote the set $\{1,2,\ldots,n\}$ by $[n]$ for simplicity.

\section{Acyclic multipartite tournaments}\label{sec:acyclic}

In this section, we compute the competition period and competition index of an acyclic multipartite tournament.

\begin{Prop}[\cite{eoh2020m}]\label{prop:acyclic-digraph}
For a digraph $D$, the following are equivalent.
\begin{itemize}
  \item[(i)] $D$ is acyclic.
  \item[(ii)] $W_{\zeta(D)}=V\left(D_{\zeta(D)}\right)\neq\emptyset$.
  \item[(iii)] $\bigcup_{i=0}^{\zeta(D)}W_i=V(D)$.
\end{itemize}
\end{Prop}

The following lemma is a stronger version of Proposition 2.3 given by Eoh {\it et al.}~\cite{eoh2020m} in the sense that, regarding the set $\mathcal{L}$ of lengths of directed walks with an initial vertex in $W_i$, $\max{\mathcal{L}} \le i$ is replaced with $\mathcal{L}=\{0,\ldots,i\}$.
As a matter of fact, their proof asserted this stronger version.

%\begin{Prop}\label{Lem:noprey}
%  Let $D$ be a digraph with $\zeta(D) \ge 1$ and $(W_0, \ldots, W_{\zeta(D)})$ be the sink sequence of $D$.
%  Then any directed walk with an initial vertex in $W_i$ has length at most $i$ in $D$ for $i=0, \ldots, \zeta(D)-1$.
%  Furthermore, if $D$ is acyclic, then even a directed walk with an initial vertex in $W_{\zeta(D)}$ has length at most $\zeta(D)$ in $D$.
%\end{Prop}

\begin{Lem}[Restatement of Proposition 2.3 in \cite{eoh2020m}]\label{lem:walk-length}
Let $D$ be a digraph with $\zeta(D)\ge1$ and $(W_0,\ldots,W_{\zeta(D)})$ be the sink sequence of $D$.
Then, for each $i=0,\ldots,\zeta(D)-1$ and each vertex $v$ in $W_i$, among the directed walks starting from $v$, there exist directed walks of lengths $0, \ldots,i$ and no directed walks of length greater than $i$.
%Then any directed walk with an initial vertex in $W_i$ has length up to $i$ in $D$ for $i=0,\ldots,\zeta(D)-1$.
Furthermore, if $D$ is acyclic, then the statement is true for even $i=\zeta(D)$.
\end{Lem}

\begin{Thm}\label{acyclic-digraph-properties}
  Let $D$ be an acyclic digraph and $\left(W_0,\ldots,W_{\zeta(D)}\right)$ be its sink sequence.
   Then, for $\zeta(D)\ge1$, we have the following:
  \begin{itemize}
    \item[(i)] $C^m(D)$ is an empty graph for any integer $m>\zeta(D)$;
    \item[(ii)] {\rm cperiod}$(D)=1$;
    \item[(iii)] if $|W_i|=1$ for some integer $i\in\{0,\ldots,{\zeta(D)}-1\}$, then $C^{\zeta(D)}(D)$ is the union of the complete graph with the vertex set $W_{\zeta(D)}$ and the empty graph with the vertex set $V(D)\setminus W_{\zeta(D)}$;
%    \item[(iii')] $C^{\zeta(D)}(D)$ is the clique formed by $W_{\zeta(D)}$  together with isolated vertices in $V(D)\setminus W_{\zeta(D)}$ if there is an integer $i\in\{0,\ldots,{\zeta(D)}-1\}$ such that $|W_i|=1$;
    \item[(iv)] {\rm cindex}$(D)\le\zeta(D)+1$ where the equality holds if $|W_{\zeta(D)}|>|W_i|$ for some integer $i\in\{0,\ldots,{\zeta(D)}-1\}$.
  \end{itemize}
\end{Thm}
\begin{proof}
Take an integer $m>\zeta(D)$.
Since $D$ is acyclic, $\bigcup_{i=0}^{\zeta(D)}W_i=V(D)$ by Proposition \ref{prop:acyclic-digraph} and so no vertex in $D$ has an $m$-step prey in $D$ by Lemma \ref{lem:walk-length}.
Therefore $C^m(D)$ is an empty graph and so the statement (i) is true.
Then, by the definition of competition period, the statement (ii) is immediately true.

To show the statement (iii), suppose that $|W_i|=1$ for some integer $i\in\{0,\ldots,{\zeta(D)}-1\}$.
Let $W_i=\{w\}$.
Since every vertex in $W_j$ has at least one out-neighbor in $W_{j-1}$ for each integer $1 \le j \le \zeta(D)$, there is a directed walk of length $\zeta(D)-i$ from $v$ to $w$ for each vertex $v\in W_{\zeta(D)}$.
Moreover, there is a directed walk of length $i$ from $w$ to a vertex $u\in W_0$.
By concatenating those directed walks, we obtain a directed walk of length $\zeta(D)$ from each vertex in $W_{\zeta(D)}$ to $u$.
Thus $W_{\zeta(D)}$ forms a clique in $C^{\zeta(D)}(D)$.
By Proposition~\ref{prop:acyclic-digraph}(iii), every vertex in $V(D)\setminus W_{\zeta(D)}$ belongs to $W_j$ for some $j \in \{0, \ldots, \zeta(D)-1\}$.
Therefore every vertex in $V(D)\setminus W_{\zeta(D)}$ is isolated in $C^{\zeta(D)}(D)$ by Lemma \ref{lem:walk-length}.
Thus the statement (iii) is true.

The inequality {\rm cindex}$(D)\le\zeta(D)+1$ immediately follows from (i).
To figure out when the equality holds, suppose that $|W_{\zeta(D)}|>|W_i|$ for some integer $i\in\{0,\ldots,{\zeta(D)}-1\}$.
As we have observed above, there are at least $|W_{\zeta(D)}|$ directed walks starting from distinct vertices in $W_{\zeta(D)}$ to a vertex in $W_{i}$.
Since $|W_{\zeta(D)}|>|W_i|$, there are at least
two directed walks terminating at the same vertex in $W_i$ by the pigeonhole principle.
Then the origins of those directed walks form a clique in $C^{\zeta(D)}(D)$.
Thus $C^{\zeta(D)}(D)$ is not an empty graph.
Hence, by the definition of competition index and (i), {\rm cindex}$(D)=\zeta(D)+1$.
\end{proof}

It is easy to check that an acyclic multipartite tournament has sink elimination index $\zeta(D)\ge1$.

\begin{Thm}\label{k-partite tournament if and only if conditon}
  For an integer $k\ge2$, let $D$ be an acyclic $k$-partite tournament with a $k$-partition $(V_1,\ldots,V_k)$. Then $\left(W_0,\ldots,W_{\zeta(D)}\right)$ is the sink sequence of $D$ if and only if $(W_0,\ldots,W_{\zeta(D)})$ is a partition of $V(D)$ satisfying the following:
  \begin{itemize}
   \item[(i)] for each $i = 0,\ldots,\zeta(D)$, $W_i$ is a subset of a partite set of $D$;
   \item[(ii)] if there is an arc from a vertex in $W_j$ to a vertex in $W_i$ for some $i,j \in \{0,\ldots, \zeta(D)\}$, then $i < j$ and $W_i$ and $W_j$ are included in different partite sets.
 %
%        for each arc $(u,v)$ in $D$, $u \in W_j$ and $v \in W_i$ for some $W_i$ and $W_j$ included in different partite sets with $0 \le i < j \le \zeta(D)$.
%   every arc in $D$ has its tail in $W_{j}$ and its head in $W_i$ for nonnegative integers $i$ and $j$ with $i<j \le \zeta(D)$ where $W_i$ and $W_j$ are included in different partite sets;
   \item[(iii)] for each $i = 0,\ldots,\zeta(D)-1$,
   there is an arc from each vertex in $W_{i+1}$ to each vertex in $W_i$.
  \end{itemize}
\end{Thm}
\begin{proof} Suppose that $\left(W_0,\ldots,W_{\zeta(D)}\right)$ is the sink sequence of $D$.
Let $(D_0,\ldots,D_{\zeta(D)})$ be the digraph sequence associated with it.
Suppose, to the contrary, that there are two vertices $u,v\in W_i$ for some $i\in\{0,\ldots,\zeta(D)\}$ such that $u\in V_p$ and $v\in V_q$ with $p\neq q$.
Since $D$ is a $k$-partite tournament, $(u,v)$ or $(v,u)$ is an arc in $D$.
Without loss of generality, we may assume that $(u,v)$ is an arc in $D$.
By the definition of $D_i$, $(u,v)$ is an arc in $D_i$, which contradicts the assumption that $u\in W_i$.
Thus, for each integer $0 \le i \le \zeta(D)$, $W_i\subseteq V_s$ for some $s\in[k]$ and so the statement (i) is true.

Suppose there is an arc from a vertex $W_j$ to a vertex $W_i$ for some $i,j\in\{0,\ldots,\zeta(D)\}$.
% by Proposition~\ref{prop:acyclic-digraph}(i)$\Leftrightarrow$(iii).
Then, by the definition of sink sequence, $i<j$.
By (i), $W_i \subseteq V_p$ and $W_j \subseteq V_q$ for some $p, q \in [k]$.
Since an arc goes from $W_j$ to $W_i$, $p \neq q$ and so the statement (ii) is true.

%Fix $i\in\{0,\ldots,\zeta(D)-1\}$. By the previous argument, $W_i\subseteq V_s$ and $W_{i+1}\subseteq V_t$ for some $s,t\in[k]$.
By definition, $W_i$ and $W_{i+1}$ are not empty sets for each $i = 0,\ldots,\zeta(D)-1$, so there exists an arc from a vertex in $W_{i+1}$ to a vertex in $W_i$.
Therefore by (ii), $W_i$ and $W_{i+1}$ are included in different partite sets.
Since $D$ is a multipartite tournament, there is an arc from each vertex in $W_{i+1}$ to each vertex in $W_i$.
Hence (iii) is true and so the `only if' part is valid.

Conversely, consider a partition $(W_0,\ldots,W_{\zeta(D)})$ of $V(D)$ satisfying (i), (ii), and (iii).
By (ii), every vertex in $W_0$ is a sink of $D$. Suppose that there is a sink $v$ of $D$ in $V(D)\setminus W_0$.
Then, since $\bigcup_{j=0}^{\zeta(D)}W_j = V(D)$, $v\in W_i$ for some integer $i\in[\zeta(D)]$.
By (iii), $W_i$ and $W_{i-1}$ are included in distinct partite sets. Since $D$ is a multipartite tournament and (ii) is true, there is an arc from $v$ to a vertex in $W_{i-1}$, which contradicts the choice of $v$.
Therefore $W_0$ is the set of sinks in $D$.
By applying a similar argument to $W_1$ of the subdigraph $D_1$ induced by $V(D)\setminus W_0$, we may show that $W_1$ is the set of sinks of $D_1$.
Inductively, we may show that $W_i$ is the set of sinks of the subdigraph of $D$ induced by $V(D) \setminus \bigcup_{j=0}^{i-1}W_j$ for $i=2,\ldots,\zeta(D)$.
Since $(W_0,\ldots,W_{\zeta(D)})$ is a partition of $V(D)$, we may conclude that $(W_0,\ldots,W_{\zeta(D)})$ is the sink sequence of $D$ and so the `if' part is true.
%For each $s=1,\ldots,\zeta(D)$, let $W_s\subseteq V_{t_s}$ for some $t_s=1,\ldots,k$.
%For each $t\in\{1,\ldots,\zeta(D)\}$, let $\mathcal{V}_t=\bigcup_{j\in J} W_j$ where $J=\{j<t\mid W_t\subseteq V_a, W_j\not\subseteq V_a \}$.
%Then we draw the arcs from each vertex in $W_{s}$ to all vertices in $W_i$ if $i<s$ and $W_i\not\subseteq V_{t_s}$.
%Thus this directed graph is the acyclic $k$-partite tournament with the sink sequence $\left(W_0,\ldots,W_{\zeta(D)}\right)$. Hence the proof is completed.
\end{proof}

\begin{Cor}\label{cor:zeta condition}
For an integer $k\ge2$, let $D$ be an acyclic $k$-partite tournament with a $k$-partition $(V_1,\ldots,V_k)$. If $\left(W_0,\ldots,W_{\zeta(D)}\right)$ is the sink sequence of $D$, then
$\zeta(D)\ge k-1$ where the equality holds if and only if $W_i$ is a partite set of $D$ for each $i=0,\ldots,\zeta(D)$;
\end{Cor}
\begin{proof} Since $D$ is acyclic, $\bigcup_{i=0}^{\zeta(D)}W_i=V(D)$ by Proposition~\ref{prop:acyclic-digraph}.
 By Theorem~\ref{k-partite tournament if and only if conditon}(i), $W_i$ is a subset of a partite set of $D$ for each integer $0 \le i \le \zeta(D)$. If $\zeta(D)<k-1$, then there exists a partite set not containing $W_i$ for any $i\in\{0,\ldots,\zeta(D)\}$, which is impossible. Thus $\zeta(D)\ge k-1$. It is obvious that $\zeta(D)=k-1$ if and only if $\{W_0,\ldots,W_{\zeta(D)}\}=\{V_1,\ldots,V_k\}$.
\end{proof}
If $D$ is an acyclic $k$-partite tournament for an integer $k\geq 3$, then $\zeta(D) \ge 2$ by Corollary~\ref{cor:zeta condition} and we have the following result.
%By Corollary~\ref{cor:zeta condition}, if $k \ge 3$ then $\zeta(D)\ge 2$ and we have the following corollary.
\begin{Cor}\label{lem:three are different}
 Let $D$ be an acyclic $k$-partite tournament for an integer $k\ge3$ and let $\left(W_0,\ldots,W_{\zeta(D)}\right)$ be the sink sequence of $D$.
   Then there is an integer $i\in\{0,\ldots,\zeta(D)-2\}$ such that $W_{i}$ and $W_{i+2}$ are included in different partite sets.
\end{Cor}
\begin{proof}
Let $(V_1,\ldots, V_k)$ be a $k$-partition of $D$.
By Theorem~\ref{k-partite tournament if and only if conditon}(iii), $W_{0}$ and $W_{1}$ are included in different partite sets.
Without loss of generality, we may assume that $W_{0}\subseteq V_1$ and $W_{1}\subseteq V_2$.
Suppose, to the contrary, that $W_{i}$ and $W_{i+2}$ are included in the same partite set for each integer $0 \le i \le \zeta(D)-2$.
Then
\begin{equation*}
\bigcup_{0 \le i \le {\zeta(D)}/2}{W_{2i}} \subseteq V_1 \quad \mbox{ and } \quad \bigcup_{0 \le i \le ({\zeta(D)}-1)/2}{W_{2i+1}} \subseteq V_2.
\end{equation*}
Therefore $\bigcup_{i=0}^{\zeta(D)}W_i\subseteq V_1\cup V_2$.
Since $k\ge 3$, $V_3\neq \emptyset$.
Thus $\bigcup_{i=0}^{\zeta(D)}W_i\subsetneq V_1\cup V_2\cup V_3\subseteq V(D)$, which contradicts Proposition~\ref{prop:acyclic-digraph}(iii).
%By Proposition \ref{acyclic-digraph-properties}, $V(D) = w_1\cup w_2$ which contradicts
\end{proof}

\begin{Lem}\label{Lem:directed path}
  Let $D$ be an acyclic $k$-partite tournament for an integer $k\ge3$ and let $\left(W_0,\ldots,W_{\zeta(D)}\right)$ be the sink sequence of $D$.
  For each integer $1\le s \le \zeta(D)$, if there exist integers $p, q \in \{0,\ldots,s\}$ with $p>q$ such that $W_{p}$ and $W_{q}$ are included in different partite sets, then there exists a directed path of length $s-p+q+1$ from each vertex in $W_s$ to each vertex in $W_0$.
\end{Lem}

\begin{proof}
Fix $s\in \{1,\ldots,\zeta(D)\}$.
Suppose that there exist integers $p, q \in \{0,\ldots,s\}$ with $p>q$ such that $W_{p}$ and $W_{q}$ are included in different partite sets.
Take vertices $u\in W_{s}$ and $x\in W_0$.
Now we take two vertices $v\in W_{p}$ and $w \in W_q$ so that $u=v$ if $p=s$ and $w=x$ if $q=0$.
  Since $D$ is a multipartite tournament, every vertex in $W_{i-1}$ is an out-neighbor of each vertex in $W_i$ for each integer  $1\le i \le \zeta(D)$.
Therefore, there exist a $(u,v)$-directed path $P$ of length $s-p$ and a $(w,x)$-directed path $Q$ of length $q$.
Since $W_{p}$ and $W_{q}$ are included in different partite sets, $v$ and $w$ are linked by an arc.
Then, by Theorem~\ref{k-partite tournament if and only if conditon}(ii), $(v,w)$ is an arc of $D$.
Thus $P\to Q$ is a $(u,x)$-directed path of length $(s-p)+1+q$.
Since $u$ and $x$ were arbitrarily chosen, the statement is true.
\end{proof}

We recall that if $D$ is an acyclic $k$-partite tournament for an integer $k\geq 3$, then $\zeta(D) \ge 2$.

\begin{Thm}\label{Thm:acyclic k-partite tournament}
  For an integer $k\geq 3$, let $D$ be an acyclic $k$-partite tournament with the sink sequence $\left(W_0,\ldots ,W_{\zeta(D) }\right)$.
  Then the following are true:
  \begin{itemize}
    \item[(i)] $C^{\zeta(D)}(D)$ is the union of the complete graph with the vertex set $W_{\zeta(D)}$ and the empty graph with the vertex set $V(D)\setminus W_{\zeta(D)}$;
%      $K_{\left\vert W_{\zeta(D)}\right\vert}\cup I_{\left\vert \mathcal{U}_{\zeta(D)}\right\vert }$;
%    \item[(ii)] $C^{\zeta(D)-1}(D)$ is isomorphic to $K_{\left\vert W_{\zeta(D)}\cup W_{\zeta(D)-1}\right\vert}\cup I_{\left\vert \mathcal{U}_{\zeta(D)-1}\right\vert }$;
          \item[(ii)] $C^{\zeta(D)-1}(D)$ is the union of the complete graph with the vertex set $W_{\zeta(D)}\cup W_{\zeta(D)-1}$ and the empty graph with the vertex set $\bigcup _{i=0}^{\zeta(D)-2}W_{i}$;
%
%          isomorphic to $K_{\left\vert W_{\zeta(D)}\cup W_{\zeta(D)-1}\right\vert}\cup I_{\left\vert V(D)\setminus (W_{\zeta(D)}\cup W_{\zeta(D)-1})\right\vert }$;
   % \item[(iii)] $C^{m}(D)$ is isomorphic to $K_{\left\vert V(D)\setminus \mathcal{U}_m\right\vert}\cup I_{\left\vert \mathcal{U}_m\right\vert }$ for each $m\in[\zeta(D)-2]$ for which $D\left[\mathcal{U}_{m+1}\right]$ is a multipartite tournament with at least three partite sets;
    \item[(iii)] if $\left\vert W_{\zeta(D)}\right\vert \geq 2$, then ${\rm cindex}(D)=\zeta(D)+1$, otherwise ${\rm cindex}(D)=\zeta(D)$.
  \end{itemize}
\end{Thm}
\begin{proof} %We first note that every vertex in $W_{i-1}$ is an out-neighbor of each vertex in $W_i$ for each $i\in[\zeta(D)]$.
%Therefore,
%\begin{itemize}
%  \item[(\dag)] for any two vertices $u\in W_{i}$ and $v\in W_{j}$, there exists a $(u,v)$-directed path of length $i-j$ for each $i,j\in[\zeta(D)]$ with $i>j$.
%\end{itemize}
Take a vertex $z \in W_0$.
By Theorem~\ref{k-partite tournament if and only if conditon}(iii),
$W_0$ and $W_1$ are included in different partite sets.
Then, since $0 < 1 \le \zeta(D)-1$, by Lemma~\ref{Lem:directed path}, there exist
\begin{itemize}
  \item[(a)] a directed path of length $\zeta(D)$ from any vertex in $W_{\zeta(D)}$ to $z$ and
  \item[(b)] a directed path of length $\zeta(D)-1$ from any vertex in $W_{\zeta(D)-1}$ to $z$.
\end{itemize}
By (a), $W_{\zeta(D)}$ forms a clique in $C^{\zeta(D)}(D)$.
By Lemma \ref{lem:walk-length}, every vertex in $\bigcup _{i=0}^{\zeta(D)-1}W_{i}$ is isolated in $C^{\zeta(D)}(D)$.
By Proposition~\ref{prop:acyclic-digraph}, $\bigcup _{i=0}^{\zeta(D)-1}W_{i} = V(D)\setminus W_{\zeta(D)}$ and so the statement (i) is true.

%By Lemma~\ref{Lem:directed path}, there is a directed path of length $\zeta(D)-1$ from each vertex in $W_{\zeta(D)-1}$ to $z$.
Since $D$ is a $k$-partite tournament with $k \ge 3$, there is an integer $i \in \{0,\ldots, \zeta(D)-2\}$ such that  $W_{i}$ and $W_{i+2}$ are included in different partite sets by Corollary \ref{lem:three are different}.
Thus, by Lemma~\ref{Lem:directed path}, there is a directed path of length $\zeta(D)-1$ from each vertex in $W_{\zeta(D)}$ to $z$.
Hence, by (b), $z$ is a $(\zeta(D)-1)$-step common prey of each vertex in $W_{\zeta(D)} \cup W_{\zeta(D)-1}$ and so $W_{\zeta(D)} \cup W_{\zeta(D)-1}$ forms a clique in $C^{\zeta(D)-1}(D)$.
By Lemma \ref{lem:walk-length}, every vertex in $\bigcup _{i=0}^{\zeta(D)-2}W_{i}$ is isolated in $C^{\zeta(D)-1}(D)$ and so, by Proposition~\ref{prop:acyclic-digraph}, the statement (ii) is true.

Suppose $\left\vert W_{\zeta(D)}\right\vert \geq 2$.
Then, by the statement (i), $C^{\zeta(D)}(D)$ is not empty.
Since $C^m(D)$ is empty for each integer $m>\zeta(D)$ by Theorem~\ref{acyclic-digraph-properties}(i), ${\rm cindex}(D)$ is $\zeta(D)+1$.
Now suppose $\left\vert W_{\zeta(D)}\right\vert \le 1$.
Then, since $D$ is acyclic, $\left\vert W_{\zeta(D)}\right\vert = 1$
and so, by the statements (i) and (ii), $C^{\zeta(D)}(D)$ is empty and $C^{\zeta(D)-1}(D)$ is not empty, respectively.
Since $C^m(D)$ is empty for each integer $m>\zeta(D)$ by Theorem~\ref{acyclic-digraph-properties}(i), ${\rm cindex}(D)$ is $\zeta(D)$.
Therefore the statement (iii) is true.
\end{proof}

\section{Multipartite tournaments with sinks and directed cycles}\label{sec:directed cycle}

If a digraph $D$ is acyclic, then $C^m(D)$ is empty for any integer $m > \zeta(D)$ and so the competition period of $D$ is $1$ (Theorem~\ref{acyclic-digraph-properties}).
In addition, Cho and Kim~\cite{cho2004competition} showed that a digraph without sinks has competition period $1$.
In this vein, this section studies the competition period of a multipartite tournament having a sink and a directed cycle.
By the way, Eoh {\it et al.}~\cite{eoh2020m} showed that if $C^M(D)$ is an empty graph for a digraph $D$ and a positive integer $M$, then so is $C^m(D)$ for any positive integer $m \ge M$.
Therefore, if $C^M(D)$ is an empty graph for a digraph $D$ and a positive integer $M$, then the competition period of $D$ is $1$.
%Thus it is meaningful to determine the competition period of a digraph $D$ such that $D$ has a sink and a directed cycle, and
%$C^m(D)$ is not empty for every positive integer $m$.
As a matter of fact, in the case of a multipartite tournament $D$, having a sink and a directed cycle guarantees that $C^m(D)$ is not empty for every positive integer $m$ by the following proposition.

\begin{Prop}\label{Prop:notempty}
  If a multipartite tournament $D$ has a sink and a directed cycle, then $C^m(D)$ is not empty for every positive integer $m$.
\end{Prop}

\begin{proof}
    Suppose that a $k$-partite tournament $D$ for an integer $k \ge 2$ has a sink $x$ and a directed cycle $C:=v_0v_1 \cdots v_{l-1} v_0$ for some integer $l \ge 3$.
    Let $(V_1,\ldots,V_k)$ be a $k$-partition of $D$.
  Without loss of generality, we may assume that $x \in   V_1$.
  Since $V_1$ is a partite set of $D$, there exist at least two vertices of $V(C) \setminus V_1$.
  Let $v_i$ and $v_j$ be vertices of $V(C)\setminus V_1$ for some two distinct integers $i,j \in \{0,1,\ldots,l-1\}$.
  Since $D$ is a $k$-partite tournament and $x$ is a sink, $(v_i,x)$ and $(v_j,x)$ are arcs of $D$.
  Then $v_i$ and $v_j$ are adjacent in $C(D)$.
  Since $v_i$ and $v_j$ are on $C$, $v_{i-m+1}$ and $v_{j-m+1}$ are adjacent in $C^m(D)$ for every positive integer $m$ (all the subscripts are reduced to modulo $l$).
  Hence $C^m(D)$ is not empty for every positive integer $m$.
\end{proof}

In the following, we shall show that if a multipartite tournament $D$ has a sink and a directed cycle, then the competition period of $D$ is at most three, which is our main result.
To do so, we need several theorems and lemmas.

\begin{Thm}[\cite{goddard1991multipartite}]\label{thm:goddard1991multipartite}
  Let $D$ be a $k$-partite tournament with $k\ge 3$.
  Then $D$ contains a directed cycle of length $3$ if and only if there exists a directed cycle in $D$ which contains vertices from at least three partite sets.
\end{Thm}

Given a digraph $D$, we call a directed cycle of length at least $4$ in $D$, a \textit{directed hole} if it is an induced subdigraph of $D$.

We note that a directed hole of length $4$ in a multipartite tournament $D$ is of the form $v_0 \to v_1 \to v_2 \to v_3 \to v_0$ such that $\{v_0,v_2\} \subseteq X$ and $\{v_1,v_3\} \subseteq Y$ for some distinct partite sets $X$ and $Y$ of $D$.

Given a multipartite tournament $D$, if a directed walk contains vertices which induce a directed cycle of length $3$ in $D$ (resp.\ a directed hole of length $4$ in $D$), then we say that it is of \emph{Type~1} (resp.\ \emph{Type~2}).

\begin{Lem}\label{lem:34cycle1}
Let $D$ be a multipartite tournament of order $n \ge 3$.
  Then any directed walk of length at least $n$ in $D$ is of Type~1 or Type~2.
\end{Lem}

\begin{proof}
  Let $Q$ be a directed walk of length at least $n$ in $D$.
  Since the length of $Q$ is greater than or equal to the number of vertices of $D$, $Q$ contains a directed cycle $C:= u_0 \to u_1 \to u_2 \to \cdots \to u_{l-1} \to u_0$ for an integer $l \ge 3$.

  {\it Case 1.} There are at least $3$ vertices on $C$ which belong to distinct partite sets in $D$.
We may apply Theorem~\ref{thm:goddard1991multipartite} to the multipartite tournament induced by $V(C)$ to conclude that there is a directed cycle of length $3$ all of whose vertices are on $C$.
Then it is easy to check that $Q$ is of Type~1.

  {\it Case 2.}  There exist two distinct partite sets $X$ and $Y$ in $D$ such that $V(C) \subseteq X \cup Y$.
  Then $l$ is even, so $l \ge 4$.
  Without loss of generality, we may assume that a vertex on $C$ with an even index belongs to $X$ and a vertex on $C$ with an odd index belongs to $Y$.
  Since $D$ is a $k$-partite tournament, there is an arc between $u_i$ and $u_{l-i-1}$ for each integer $0 \le i \le l/2 - 1$.
Since $(u_{l-1},u_0) \in A(D)$ and $C$ is a directed cycle, $(u_{i},u_{l-i-1})$ is an arc in $D$ for some $i \in \{1,\ldots, l/2 -1\}$.
We may regard $i$ as the smallest index among $1, \ldots, l/2 -1$ such that $(u_{i},u_{l-i-1})$ is an arc in $D$.
  Then $(u_{l-i},u_{i-1}) \in A(D)$ and $C' := u_{i-1} \to u_i \to u_{l-i-1} \to u_{l-i} \to u_{i-1}$ is a directed hole of length $4$ in $D$.
  Thus $Q$ is of Type~2.
 \end{proof}

We note that a multipartite tournament has a sink if and only if the sink elimination index is greater than or equal to one.
In the following five results prior to our main result, we examine lengths of directed walks from a vertex in $D_{\zeta(D)}$ to a vertex in $\bigcup_{I=0}^{\zeta(D)-1}W_i$ where $(W_0,\ldots,W_{\zeta(D)})$ is the sink sequence of a multipartite tournament $D$ with a directed cycle and $\zeta(D) \ge 1$.

\begin{Lem}\label{lem:3cycle}
 Let $D$ be a multipartite tournament.
  Suppose that there exists a $(u,v)$-directed walk of Type~1 of length $\ell$ for some vertices $u$ and $v$.
Then there is a $(u,v)$-directed walk of length $\ell+3m$ for each nonnegative integer $m$.
\end{Lem}

\begin{proof}
  Let $Z$ be a $(u,v)$-directed walk of length $\ell$ that contains three vertices which induce a directed cycle $C$ of length $3$ in $D$.
  Let $x$ be a vertex on $C$, $Z_1$ be a $(u,x)$-section of $Z$, and $Z_2$ be the $(x,v)$-section of $Z$ obtained by cutting $Z_1$ away from $Z$.
  We may assume that the sequence representing $C$ starts at $x$.
  Then, for a nonnegative integer $m$, we may create a $(u,v)$-directed walk of length $\ell + 3m$ in such a way that we traverse $Z_1$, $C$ as many as $m$ times, and then $Z_2$.
\end{proof}

\begin{Lem}\label{lem:4cycle}
Let $D$ be a multipartite tournament
with the sink elimination index $\zeta(D) \ge 1$ and
the sink sequence $\left(W_0,\ldots,W_{\zeta(D)}\right)$.
  For a vertex $u$, suppose that there exists a $(u,w)$-directed walk $Z$ of Type~2 of length $\ell$ for some vertex $w$ such that (i) $w$ belongs to $\bigcup_{i=0}^{\zeta(D)-1}W_i$ or
  (ii) there exists a directed hole of length $4$ such that its four vertices are on $Z$ and $w$ does not belong to any of two partite sets which the four vertices belong to.
Then there exists a positive integer $N$ such that for each integer $m \ge N$, there is a $(u,w)$-directed walk of length $\ell+2m$.
\end{Lem}

\begin{proof}
Since $Z$ is of Type~2, there are four vertices on $Z$ which induce a directed hole $H:=v_0 \to v_1 \to v_2 \to v_3 \to v_0$ of length $4$ in $D$.
 Then there exist two distinct partite sets $X$ and $Y$ of $D$ such that $\{v_0,v_2\} \subseteq X$ and $\{v_1,v_3\} \subseteq Y$.
If the case (ii) of the lemma statement happens, we may assume that $H$ is the hole mentioned in the case.
Let $Z_1$ be a $(u,v_0)$-section of $Z$, and $Z_2$ be the $(v_0,w)$-section of $Z$ obtained by cutting $Z_1$ away from $Z$.
   We denote the lengths of $Z_1$ and $Z_2$ by $\ell_1$ and $\ell_2$, respectively.

  {\it Case 1.} There exists a vertex $z$ on $Z_2$ which does not belong to $X \cup Y$.
  Let $Z_3$ be a $(v_0,z)$-section of $Z_2$, and $Z_4$ be the $(z,w)$-section of $Z_2$ obtained by cutting $Z_3$ away from $Z_2$.
   We denote the lengths of $Z_3$ and $Z_4$ by $\ell_3$ and $\ell_4$, respectively.
  We consider the two subcases: $(v_i,z) \in A(D)$ for each $i =0,1,2,3$; there is an arc $(z,v)$ for some vertex $v$ on $H$.
  Suppose $(v_i,z) \in A(D)$ for each $i =0,1,2,3$ and fix a positive integer $\alpha$.
   Then, traverse the directed walk $Z_1$, go around $H$ as many times as desired, depart it at an appropriate vertex on $H$ to reach $z$ by an arc, and traverse $Z_4$.
   This creates a $(u,w)$-directed walk of length $\ell_1 + \alpha + \ell_4$.

  Now suppose that there is an arc $(z,v)$ for some vertex $v$ on $H$.
  Let $D^*$ be the subdigraph of $D$ induced by $V(Z_3) \cup V(H)$.
  Then $D^*$ contains vertices from at least three partite sets.
  Since $D^*$ is a multipartite tournament,
    $D^*$ contains a directed cycle $C$ of length $3$ by Theorem~\ref{thm:goddard1991multipartite}.
  Moreover, the directed walk $Z_3$, the arc $(z,v)$, and $(v,v_0)$-section of $H$ form a closed directed walk $Q$ in $D^*$.
    Then $V(Q) \cup V(H) = V(D^*)$ and $V(Q) \cap V(H) \neq \emptyset$, so $D^*$ is strongly connected.
    Since $D^*$ contains the directed cycles $C$ of length $3$ and $H$ of length $4$, we may conclude that $D^*$ is primitive.
    Thus, for any $\beta \ge \exp(D^*)$, there is a $(v_0,z)$-directed walk of length $\beta$ and so there is a $(u,w)$-directed walk of length $\ell_1 + \beta + \ell_4$.

    Since $\ell_1 + \ell_4 \le \ell$ and $\alpha$ and $\beta$ were arbitrarily chosen among the positive integers bounded below, we have shown in both subcases that there exists a sufficiently large $N$ such that there is a $(u,w)$-directed walk of length $\ell+2m$ for each integer $m \ge N$.

  {\it Case 2.} Every vertex on $Z_2$ belongs to $X \cup Y$.
  Then the case (i) of the lemma statement happens, that is, $w \in \bigcup_{i=0}^{\zeta(D)-1}W_i$.
  Let $j=0$ if $w \in Y$ and $j=1$ if $w \in X$.
  Then $2m+j+1$ and $\ell_2$ have the same parity.
  For, $\ell_2$ is odd if $w \in Y$ and $\ell_2$ is even if $w \in X$ since $v_0 \in X$ and $Z_2$ is a $(v_0,w)$-directed walk whose vertices belong to $X \cup Y$.
  Furthermore, $w$ and any of $v_j, v_{j+2}$ belong to distinct partite sets.
  Since $w\in\bigcup_{i=0}^{\zeta(D)-1}W_i$ and $\{v_0,v_1,v_2,v_3\}\subseteq V(D_{\zeta(D)})$, there are arcs $(v_j,w)$ and $(v_{j+2},w)$ in $D$.
     Now let
     \[
     \Theta_i = Z_1 \to H_i \to H' \to w
     \]
     %
%     \[ \Theta_i =
%     \begin{cases}
%       Z_1 \to H_i \to w , & \mbox{if } w \in Y; \\
%       Z_1 \to H_i \to v_1 \to w, & \mbox{if } w \in X,
%     \end{cases}
%     \]
     where $H'$ denotes the $(v_0,v_j)$-section of $H$ and $H_i$ means the directed walk obtained by going around $H$ $i$ times, for a nonnegative integer $i$.
     Then $\Theta_i$ is a $(u,w)$-directed walk in $D$.
     For a nonnegative integer $i$, we denote by $\Lambda_i$ the directed walk obtained from $\Theta_i$ by replacing the arc $(v_j,w)$ with the directed path $v_j \to v_{j+1} \to v_{j+2} \to w$.
    Then the lengths of $\Theta_i$ and $\Lambda_i$ are $\ell_1 + 4i + j+1$ and $\ell_1 + 4i + j+3$, respectively, for each nonnegative integer $i$.
    Accordingly, we have shown that there is a $(u,w)$-directed walk of length $\ell_1 + 2m+j+1$ for each nonnegative integer $m$.
    Since $2m+j+1$ and $\ell_2$ have the same parity, we have actually shown that there exists $(u,w)$-directed walk of length $\ell + 2m$ for each nonnegative integer $m$.
\end{proof}

Let $p_1,\ldots, p_t$ be positive integers with
$\gcd(p_1,\ldots,p_t)=1$.
The {\em Frobenius number} of
$p_1,\ldots, p_t$ is the largest integer $b$ for which the
Frobenius equation $$p_1x_1+\cdots+p_tx_t=b$$ has no nonnegative
integer solution $(x_1,\ldots,x_t)$.
The number $b$ is denoted by $F(p_1,\ldots,p_t)$.

\begin{Lem}\label{lem:frobenius}
Let $D$ be a multipartite tournament
with the sink elimination index $\zeta(D) \ge 1$ and
the sink sequence $\left(W_0,\ldots,W_{\zeta(D)}\right)$.
  For a vertex $u \in V(D_{\zeta(D)})$, suppose that there exists a $(u,w)$-directed walk of length $\ell$ for some vertex $w\in \bigcup_{i=0}^{\zeta(D)-1}W_i$ such that its sequence contains both a directed cycle of length $3$ and a directed hole of length $4$.
Then there exists a positive integer $N$ such that for each integer $m \ge N$, there is a $(u,w)$-directed walk of length $\ell+m$.
  \end{Lem}

\begin{proof}
  Let $Z$ be a $(u,w)$-directed walk of length $\ell$ whose sequence contains a directed cycle $C$ of length $3$ and a directed hole $H$ of length $4$.
  If (i) the sequence of $C$ appears before that of $H$ on $Z$, then there exists a vertex on $H$ such that, on $Z$, all the vertices on $C$ appear before it and all the vertices on $H$ appear after it.
  We denote such a vertex by $x$.
  Suppose that (ii) the sequence of $H$ appears before that of $C$ on $Z$.
    Since the vertices on $H$ belong to two distinct partite sets and the vertices on $C$ belong to three distinct partite sets, there exists a vertex $y$ on $C$ which does not belong to any of two partite sets which contain the vertices on $H$.

  Now we denote by $v$ the vertex $x$ if (i) happens, and the vertex $y$ if (ii) occurs.
  Let $Z_1$ be a $(u,v)$-section of $Z$, and $Z_2$ be the $(v,w)$-section of $Z$ obtained by cutting $Z_1$ away from $Z$.
   We denote the lengths of $Z_1$ and $Z_2$ by $\ell_1$ and $\ell_2$, respectively.
    Fix a nonnegative integer $t$.
     Since $F(2,3)=1$, there exist nonnegative integers $m_1, m_2$ such that $t = 3m_1 + 2m_2 -2 $.

     Consider the case (i).
     Then $v=x$.
     Since $Z_1$ is a $(u,v)$-directed walk of Type~1, by Lemma~\ref{lem:3cycle}, there exists a $(u,v)$-directed walk $Q_1$ of length $\ell_1 + 3m_1$.
   Since $Z_2$ is a $(v,w)$-directed walk of Type~2, by Lemma~\ref{lem:4cycle}, there exists a positive integer $N'$ such that there is a $(v,w)$-directed walk $Q_2$ of length $\ell_2+2(m_2 + N')$.
   Then $Q_1 \to Q_2$ is a $(u,w)$-directed walk of length
   \[
   (\ell_1 + 3m_1) + (\ell_2 + 2(m_2+N')) = \ell + 2N' +2 + (3m_1 + 2m_2 -2)= \ell + 2N'+2 + t.
   \]

    Consider the case (ii).
    Then $v=y$.
   By Lemma~\ref{lem:4cycle}, there exists a positive integer $N''$ such that there is a $(u,v)$-directed walk $Q'_1$ of length $\ell_1+2(m_2 + N'')$.
    By applying Lemma~\ref{lem:3cycle} to the $(v,w)$-directed walk $C \to Z_2$, there exists a $(v,w)$-directed walk $Q'_2$ of length $3+ \ell_2 + 3m_1$.
   Then $Q'_1 \to Q'_2$ is a $(u,w)$-directed walk of length
   \[
    (\ell_1 + 2(m_2+N'')) + (3+ \ell_2 + 3m_1)   = \ell + 2N'' +5 + (3m_1 + 2m_2 -2)= \ell + 2N''+5 + t.
   \]
    Let
   \[
   N=
   \begin{cases}
     2N'+2, & \mbox{if } v=x; \\
     2N''+5, & \mbox{if } v=y.
   \end{cases}
   \]
   Then, since $t$ was chosen as an arbitrary nonnegative integer, we have shown that, for each integer $m \ge N$, there is a $(u,w)$-directed walk of length $\ell + m$.
\end{proof}

 \begin{Lem}\label{lem:both34}
 Let $D$ be a multipartite tournament
with a directed cycle, the sink elimination index $\zeta(D) \ge 1$, and
the sink sequence $\left(W_0,\ldots,W_{\zeta(D)}\right)$.
  For a vertex $u \in V(D_{\zeta(D)})$, suppose that there exist a $(u,w_1)$-directed walk $Q_1$ of Type~1 and a $(u,w_2)$-directed walk $Q_2$ of Type~2 for some vertices $w_1,w_2\in\bigcup_{i=0}^{\zeta(D)-1}W_i$.
  Then there exists a positive integer $N$ such that
   there is a $(u,w)$-directed walk of length $\ell$ for each vertex $w \in  \bigcup_{i=0}^{\zeta(D)-1}W_i$ and any integer $\ell \ge N$.
\end{Lem}

\begin{proof}
Take $w \in  \bigcup_{i=0}^{\zeta(D)-1}W_i$.
Let $C$ be a directed cycle of length $3$ in $D$ induced by three vertices in $Q_1$ and
 $H$ be a directed hole of length $4$ in $D$ induced by four vertices in $Q_2$.
 Then there are three distinct partite sets of $D$ to which the vertices on $C$ belong.
  Since the vertices on $H$ belong to two distinct partite sets, there must be an arc linking a vertex $x$ on $C$ and a vertex $y$ on $H$.
  Moreover, there exist an arc from a vertex  $y'$ on $H$ to $w$ and an arc from a vertex $x'$ on $C$ to $w$ since $w \in  \bigcup_{i=0}^{\zeta(D)-1}W_i$ and the vertices on directed cycles belong to $V(D_{\zeta(D)})$.

If $(x,y) \in A(D)$ (resp.\ $(y,x) \in A(D)$), then a $(u,x)$-section of $Q_1$, $C$ starting at $x$, the arc $(x,y)$, $H$ starting at $y$, the $(y,y')$-section of $H$, and the arc $(y',w)$ (resp.\ a $(u,y)$-section of $Q_2$, $H$ starting at $y$, the arc $(y,x)$, $C$ starting at $x$, the $(x,x')$-section of $C$, and the arc $(x',w)$) form a $(u,w)$-directed walk $Z_w$ whose sequence contains both a directed cycle of length $3$ and a directed hole of length $4$.
Then there exists a positive integer $N_w$ such that, for each integer $m \ge N_w$, there is a $(u,w)$-directed walk of length $\ell_w+m$ by Lemma~\ref{lem:frobenius} where $\ell_w$ is the length of $Z_w$.
If we denote by $N$ the maximum of such integers $\ell_w + N_w$, there exists a $(u,w)$-directed walk of length $\ell$ for each vertex $w \in  \bigcup_{i=0}^{\zeta(D)-1}W_i$ and each integer $\ell \ge N$.
\end{proof}

\begin{Thm}\label{thm:34cycle}
    Let $D$ be a multipartite tournament
with a directed cycle, the sink elimination index $\zeta(D) \ge 1$, and
the sink sequence $\left(W_0,\ldots,W_{\zeta(D)}\right)$.
    Then, for each vertex $w\in \bigcup_{i=0}^{\zeta(D)-1}W_i$, one of the following properties is true:
  \begin{itemize}
  \item[(i)] for every vertex $u$ in $D_{\zeta(D)}$, there is a $(u,w)$-directed walk of Type~1;
    \item[(ii)] for every vertex $u$ in $D_{\zeta(D)}$, there is a $(u,w)$-directed walk of Type~2.
  \end{itemize}
\end{Thm}

\begin{proof}
 Suppose that $D$ has $n$ vertices and fix vertex $w \in \bigcup_{i=0}^{\zeta(D)-1}W_i$.
 Since $D$ has a directed cycle, $n \ge 3$.
 We fix $u \in V(D_{\zeta(D)})$.
  Since $D_{\zeta(D)}$ has no sinks, there is a directed walk $Q_u$ in $D_{\zeta(D)}$ of length at least $n$ starting at $u$.
  Since $D_{\zeta(D)}$ is a subdigraph of $D$, $Q_u$ is a directed walk in $D$.
  By Lemma~\ref{lem:34cycle1}, $Q_u$ is of Type~1 or Type~2.
  Then $Q_u$ contains vertices which induce a directed cycle $C_u$ in $D$ which has a length $3$ or is a directed hole of length $4$.
  Since the vertices on $C_u$ belong to at least two distinct partite sets, there must be a vertex $y_u$ on $C_u$ which belongs to a partite set distinct from the one to which $w$ belongs.
  Since $D$ is a multipartite tournament, there must be an arc linking $y_u$ and $w$.
  Yet, $w\in \bigcup_{i=0}^{\zeta(D)-1}W_i$ implies $(y_u,w) \in A(D)$.
  Then the $(u,y_u)$-section of $Q_u$, $C_u$ starting at $y_u$, and the arc $(y_u,w)$ form a $(u,w)$-directed walk of Type~1 or Type~2 in $D$.
   Thus we have shown that
  \begin{itemize}
    \item[($\star$)]  for every vertex $u$ in $D_{\zeta(D)}$, there is a $(u,w)$-directed walk of Type~1 or Type~2.
  \end{itemize}

  For every vertex $u$ in $D_{\zeta(D)}$, if there exist a $(u,w)$-directed walk of Type~1 and a $(u,w)$-directed walk of Type~2, then the theorem statement is immediately true.

  Now suppose that there exists a vertex $v$ in $D_{\zeta(D)}$ such that either there is no $(v,w)$-directed walk of Type~1 or there is no $(v,w)$-directed walk of Type~2.
  Then, by Lemma~\ref{lem:34cycle1}, every $(v,w)$-directed walk of length at least $n$ is of Type~2 or every $(v,w)$-directed walk of length at least $n$ is of Type~1.
  We assume the latter.
  Then, by $(\star)$, there exists a $(v,w)$-directed walk $Z_v$ of Type~1.
   Then $Z_v$ contains vertices which induce a directed cycle $C$ of length $3$ in $D$.
  %Then every $(u,w)$-directed walk is of Type~1 but not of Type~2.
  In the following, we will show that for every vertex $u$ in $D_{\zeta(D)}$, there is a $(u,w)$-directed walk of Type~1.
  Take a vertex $u$ in $D_{\zeta(D)}$.
  By $(\star)$, there is a $(u,w)$-directed walk of Type~1 or Type~2.
  If there is a $(u,w)$-directed walk of Type~1, then there is nothing to prove.
  Now suppose that there exists a $(u,w)$-directed walk $Z_u$ of Type~2.
  Then $Z_u$ contains vertices which induce a directed hole $H$ of length $4$ in $D$ and each vertex on $H$ belongs to one of two distinct partite sets of $D$.
  Since the vertices on $C$ belong to three distinct partite sets, there must be an arc between a vertex $x$ on $C$ and a vertex $y$ on $H$.
  Yet, since no directed walk of Type~2 starting at $v$ exists, $(y,x)$ is an arc in $D$.
  Now, a $(u,y)$-section of $Z_u$, the arc $(y,x)$, $C$ starting at $x$, an $(x,w)$-section of $Z_v$ form a $(u,w)$-directed walk of Type~1.
  Therefore we have shown that (i) is true.
  By applying a similar argument, one can show that (ii) is true in the former case.
\end{proof}

Let $D$ be a multipartite tournament
with a directed cycle, the sink elimination index $\zeta(D) \ge 1$, and
the sink sequence $\left(W_0,\ldots,W_{\zeta(D)}\right)$.
Then we say that a vertex in $\bigcup_{i=0}^{j-1}W_i$ is of {\it Type 1} (resp.\ {\it Type 2}) if (i) (resp.\ (ii)) of the above theorem is true. By the same theorem, each vertex in $\bigcup_{i=0}^{j-1}W_i$ is of Type~1 or Type~2.

\begin{Thm}\label{thm:at most 3}
  If a $k$-partite tournament $D$ has a sink and a directed cycle for an integer $k \ge 2$, then the competition period of $D$ is at most three.
  Especially, if $k=2$, then the competition period of $D$ is at most two.
\end{Thm}

\begin{proof}
  Let $D$ be a $k$-partite tournament with a sink and a directed cycle for an integer $k \ge 2$.
  Since $D$ has a sink and a directed cycle, $\zeta(D) \ge 1$.
  Let $\left(W_0,\ldots ,W_{\zeta(D) }\right)$ be the sink sequence of $D$ and $\mathcal{U}_{\zeta(D)}=\bigcup_{i=0}^{\zeta(D)-1}W_i$.
  Since $D$ has a directed cycle, $W_{\zeta(D)} = \emptyset$ by Proposition~\ref{prop:acyclic-digraph}.
  By Proposition~\ref{Prop:notempty}, $C^m(D)$ is not empty for every positive integer $m$.
  %and $\mathcal{U}_j=\cup _{i=0}^{j-1}W_{i}$ for $1\leq
  %j\leq \zeta(D)+1$.
  If $u \in \mathcal{U}_{\zeta(D)}$, then $u$ has no $m$-step prey for any integer $m \ge \zeta(D)$ by Lemma~\ref{lem:walk-length}.
  Therefore every vertex in $\mathcal{U}_{\zeta(D)}$ is isolated in $C^m(D)$ for any integer $m \ge \zeta(D)$.
   Thus it is sufficient to consider the vertices in  $V(D_{\zeta(D)})$ when determining the competition period of $D$.

   Suppose that there exist vertices $w_1$ and $w_2$ in
   $\mathcal{U}_{\zeta(D)}$ of Type~1 and Type~2, respectively.
   Fix $w \in \mathcal{U}_{\zeta(D)}$.
   Then, by Lemma~\ref{lem:both34}, for each vertex $u$ in $D_{\zeta(D)}$, there exists a $(u,w)$-directed walk of length $\ell$ for each integer $\ell \ge N_u$ for some positive integer $N_u$.
   If we denote by $N$ the maximum of such integers $N_u$, then there exists a $(u,w)$-directed walk of length $\ell$ for each vertex $u$ in $D_{\zeta(D)}$ and each integer $\ell \ge N$.
   Thus, for each integer $m \ge N$, $w$ is an $m$-step prey of each vertex $u$ in $D_{\zeta(D)}$ and so $V(D_{\zeta(D)})$ forms a clique in $C^m(D)$ and $D$ has the competition period one.

   Now it remains to consider the case that every vertex in $\mathcal{U}_{\zeta(D)}$ is only of Type~1 or every vertex in $\mathcal{U}_{\zeta(D)}$ is only of Type~2.
   We first consider the case in which every vertex in $\mathcal{U}_{\zeta(D)}$ is only of Type~1.
%   Suppose that there is a vertex $u$ in $D_{\zeta(D)}$ such that a directed walk starting at $u$ contains a directed hole $C$ of length $4$.
%   Then at least one vertex of $C$ belongs to a partite set distinct from the partite set to which the sinks belong.
%   It is easy to see the existence of a $(u,w)$-directed walk of Type $2$ and we reach a contradiction to the assumption that every sink in $D$ is of Type~1.
Suppose that there are two vertices $u_1$ and $u_2$ in $D_{\zeta(D)}$ which are adjacent in infinitely many step competition graphs of $D$.
   Then there exist a $(u_1,w)$-directed walk $Z_1$ and a $(u_2,w)$-directed walk $Z_2$ of the same length $\ell(u_1,u_2) \ge |V(D)|$ for a vertex $w$ in $D$.
   If $w \in V(D_{\zeta(D)})$, then $u_1$ and $u_2$ are adjacent in $C^m(D)$ for each integer $m \ge \ell(u_1,u_2)$ since $D_{\zeta(D)}$ has no sink.
   Now suppose that $w \in \mathcal{U}_{\zeta(D)}$.
   Since $\ell(u_1,u_2) \ge |V(D)|$, each of $Z_1$ and $Z_2$ is of Type~1 or Type~2 by Lemma~\ref{lem:34cycle1}.
   By the case assumption, there exist a $(u_1,w)$-directed walk $Z_3$ and a $(u_2,w)$-directed walk $Z_4$ of Type~1.
  Then, by considering the following three cases:
  \begin{itemize}
    \item[(a)] both of $Z_1$ and $Z_2$ are of Type~1;
    \item[(b)] both of $Z_1$ and $Z_2$ are of Type~2;
    \item[(c)] $Z_1$ and $Z_2$ are of different types
  \end{itemize}
   and by applying Lemmas~\ref{lem:3cycle} and \ref{lem:both34} to $Z_1, Z_2, Z_3$, or $Z_4$ whichever suitable,
   we may deduce one of the following:
  \begin{itemize}
   \item[(i)] for some positive integer $L(u_1,u_2)$ and any integer $m \ge L(u_1,u_2)$,
   $u_1$ and $u_2$ have an $(\ell(u_1,u_2)+3m)$-step common prey;
   \item[(ii)]  $u_1$ and $u_2$ have an $m$-step common prey for any integer $m \ge N(u_1,u_2)$ for some positive integer $N(u_1,u_2)$.
   \end{itemize}
Suppose (i) happens and $u_1$ and $u_2$ have an $\ell(u_1,u_2)+3m^*+i$ prey for some integer $m^* \ge L(u_1,u_2)$ and some $i$ in $\{1,2\}$.
Then, by repeating the above argument for $\ell(u_1,u_2)+3m^*+i$ instead of $\ell(u_1,u_2)$, we may guarantee the existence of a positive integer $L'(u_1,u_2) \ge L(u_1,u_2)$ such that $u_1$ and $u_2$ have an $(\ell(u_1,u_2)+3m^*+i + 3m)$-step common prey for any integer $m \ge L'(u_1,u_2)$.
Now $u_1$ and $u_2$ have an $(\ell(u_1,u_2)+3m)$-step common prey and an $(\ell(u_1,u_2)+3m +i)$-step common prey for any integer $m \ge L'(u_1,u_2)$.
Even if $u_1$ and $u_2$ have an $(\ell(u_1,u_2)+3m^{**}+j)$-step common prey for some $m^{**} \ge L'(u_1,u_2)$ and some $j \in \{1,2\} \setminus i$, we may apply the same argument to find a positive integer $L''(u_1,u_2) \ge L'(u_1,u_2)$ such that $u_1$ and $u_2$ have an $(\ell(u_1,u_2)+3m)$-step common prey, an $(\ell(u_1,u_2)+3m +i)$-step common prey, and an $(\ell(u_1,u_2)+3m +j)$-step common prey for any integer $m \ge L''(u_1,u_2)$.

We let $M(u_1,u_2)$ stands for one of  $L(u_1,u_2)$, $L'(u_1,u_2)$, $L''(u_1,u_2)$, $N(u_1,u_2)$, whichever appropriate.
%Then we may take $L(u_1,u_2)+3m$ as $L(u_1,u_2)$ for any nonnegative integer $m$.
%Therefore we may take $L(u_1,u_2)$ sufficiently large so that the following holds:
% \begin{itemize}
%   \item[(i$'$)] for an integer $L(u_1,u_2) \ge \ell(u_1,u_2)$ and any nonnegative integer $m$,
%   $u_1$ and $u_2$ have an $(L(u_1,u_2)+3m)$-step common prey and one of the following is true:
%   \begin{itemize}
%     \item[$\bullet$] they do not have an $(L(u_1,u_2)+3m+i)$-step common prey for any $i$ in $\{1,2\}$;
%     \item[$\bullet$] they have an $(L(u_1,u_2)+3m+i)$-step common prey for exactly one $i$ in $\{1,2\}$ and no $(L(u_1,u_2)+3m+j)$-step common prey for $j$ in $\{1,2\} \setminus \{i\}$.
%   \end{itemize}
%   \end{itemize}
Let $L$ be the maximum of $M(u_1,u_2)$ over the pairs $\{u_1,u_2\}$ in $D_{\zeta(D)}$ which are adjacent in infinitely many step competition graphs of $D$ ($L$ exists since there are at most $\binom{|V({D_{\zeta(D)}})|}{2}$ pairs to consider).
   Then it is easy to check that $C^{L+i}(D)=C^{L +3+i}(D)$ for each nonnegative integer $i$, so the competition period of $D$ is $1$ or $3$.
  Using Lemma~\ref{lem:4cycle}, one can show that the competition period of $D$ is at most two by a similar argument if every vertex in $\mathcal{U}_{\zeta(D)}$ is of Type~2, from which the `especially' part follows.
 \end{proof}

%\begin{Cor}
%  cperiod=3이 되려면 항상 3짜리만 가져야 한다.
%  cperiod=2가 되려면 2...
%  common을 바로 가지면 cperiod=1이 될 수도 있어서 역은 성립하지 않는다.
%\end{Cor}

\section{Strongly connected multipartite tournaments}\label{sec:primitive}

In this section, we study competition indices of strongly connected multipartite tournaments.
Cho and Kim~\cite{cho2004competition} showed that a digraph without sinks has competition period $1$.
In this section, we show that the competition index of a primitive digraph is at most its exponent.

\begin{Prop}\label{prop:adjacent}
 Let $D$ be a digraph without sinks.
 If two vertices are adjacent in $C^M(D)$ for a positive integer $M$, then they are also adjacent in $C^m(D)$ for any positive integer $m \ge M$.
\end{Prop}

\begin{proof}
  Let $x$ and $y$ are adjacent in $C^M(D)$.
  Then $x$ and $y$ have an $M$-step common prey $z$ in $D$.
  Since $D$ has no sinks, $z$ has an out-neighbor $w$ in $D$.
  Then $w$ is an $(M+1)$-step common prey of $x$ and $y$.
  Hence $x$ and $y$ are adjacent in $C^{(M+1)}(D)$.
  We may repeat this argument to show that $x$ and $y$ are adjacent in $C^{(M+2)}(D)$.
  In this way, we may show that $x$ and $y$ are adjacent in $C^m(D)$ for any positive integer $m \ge M$.
\end{proof}
%
%A digraph is said to be \textit{weakly connected} if its underlying graph is connected.
%
%\begin{Prop}\label{prop:weakly}
%  Let $D$ be a weakly connected digraph which is not a directed cycle.
%  If $D$ has no sinks, then the $m$-step competition graph of $D$ is not empty for every positive integer $m$.
%\end{Prop}
%
%\begin{proof}
%   Suppose that $D$ has no sinks.
%   Then there exists a directed cycle $C$ of length $l$ for some positive integer $l$.
%   Since $D$ is not a directed cycle, $C$ contains a chord or there is a vertex not on $C$.
%   In either case, there exists a vertex on $C$ which has at least two in-neighbors.
%   We note that those in-neighbors are pairwise adjacent in $C(D)$, so $C(D)$ is not empty.
%   By Proposition~\ref{prop:adjacent}, $C^m(D)$ is not empty for every positive integer $m$.
%\end{proof}
%
%The following corollary is an immediate consequence of Proposition~\ref{prop:weakly}.
%
%\begin{Cor}
%  If a multipartite tournament $D$ is not a directed $3$-cycle and has no sinks, then the $m$-step competition graph of $D$ is not empty for every positive integer $m$.
%\end{Cor}

%Eoh {\it et al.}~\cite{EKY} showed the proposition below.  By the way, their proof asserts a stronger statement as follows.
%{\color{blue} 뭔가 적어야 함}

\begin{Thm}\label{thm:primitive digraph}
Let $D$ be a primitive digraph. Then
\begin{itemize}
  \item[(i)] $C^m(D)$ is a complete graph for each integer $m\ge\exp(D)$;
  \item[(ii)] ${\rm cindex}(D)\le\exp(D)$;
  \item[(iii)] ${\rm cperiod}(D)=1$.
\end{itemize}
\end{Thm}
\begin{proof}
Since we assumed that the underlying graph of each digraph $D$ dealt in this paper is simple, $D$ contains neither a loop nor a directed cycle of length $2$.
Then, by the hypothesis that $D$ is primitive, $|V(D)| \ge 4$.
Now take two distinct vertices  $u$ and $v$ in $D$ and let $t=\exp(D)$.
Then there exist a $(u,w)$-directed walk and a $(v,w)$-directed walk of length $t$ for any vertex $w$ in $D$.
Therefore $u$ and $v$ are adjacent in $C^t(D)$.
Since $u$ and $v$ are arbitrarily chosen, $C^t(D)$ is a complete graph.
Since $D$ is primitive, $D$ does not contain a sink and so, by Proposition~\ref{prop:adjacent}, the statement (i) is true.
By the definitions of cindex and cperiod, the statements (ii) and (iii) are true.
\end{proof}

   Let $D$ be a strongly connected $k$-partite tournament of order $n$ such that the length of a longest directed cycle is $k$ for an integer $4 \le k \le n$ and take two vertices $u$ and $v$ in $D$.
   Bondy~\cite{bondy1976diconnected} showed that for an integer $k\ge3$, a strongly connected $k$-partite tournament contains a directed cycle of length $m$ for each integer $3 \le m \le k$.
   Therefore $D$ is primitive.
   On the other hand, Volkmann~\cite{volkmann2007multipartite} showed that for an integer $k\ge3$, every vertex of any strongly connected $k$-partite tournament with a longest cycle of length $k$ belongs to a directed cycle of length $m$ for each integer $3 \le m \le k$.
   Thus there is a directed cycle $C_{m}$ of length $m$ which contains the vertex $u$ for each integer $3 \le m \le k$.
   Now, since $D$ is strongly connected, there is a $(u,v)$-directed walk $W_1$ of length $l \le n-1$ in $D$.
  Since $l\le n-1$, $F(3,4,\ldots,k) +n -l > F(3,4,\ldots,k)$.
  Therefore, by concatenating the directed cycles, we obtain a closed directed walk $W_2$ of length $F(3,4,\ldots,k) +n -l$.
  Now $W_2 \rightarrow W_1$ is a $(u,v)$-directed walk of length $F(3,4,\ldots,k)+n$.
  Since $u$ and $v$ were arbitrarily chosen, $\exp(D) \le F(3,4,\ldots,k)+n$.
  It is known that $F(3,4)=5$ and $F(3,4,\ldots,k)=2$ for any integer $k \ge 5$.
  Thus
\begin{equation*}\exp(D)\le
  \begin{cases}
    5+n, & \mbox{if $k=4$;} \\
    2+n, & \mbox{otherwise.}
  \end{cases}
  \end{equation*}

Now, by Theorem~\ref{thm:primitive digraph}(ii), we have the following proposition.
\begin{Prop}
  Let $D$ be a strongly connected $k$-partite tournament of order $n$ such that the length of a longest directed cycle is $k$ for an integer $4 \le k \le n$.
  Then
  \begin{equation*}\mathrm{cindex}(D)\le
  \begin{cases}
    5+n, & \mbox{if $k=4$;} \\
    2+n, & \mbox{otherwise.}
  \end{cases}
  \end{equation*}
\end{Prop}
\section{Tournaments}\label{sec:tournaments}

A tournament with $n$ vertices, denoted by $T_n$, is a digraph resulting from
orienting the edges of a complete graph $K_n$.
The outdegree of a vertex $v_i$ in
$T_n$ is called the score of $v_i$, denoted by $s_i$.
If the vertices of an $n$-tournament
$T_n$ are labeled $v_1,v_2 ,\ldots, v_n$ so that $0\le s_1 \le s_2\le \cdots\le s_n$, then the sequence $(s_1,s_2,\ldots,s_n)$ is called the {\em{score sequence}} of $T_n$.

Let $D$ be a $k$-partite tournament with $n$ vertices. Then clearly $k$ is a positive integer with $k\le n$. One can easily see that $k=n$ if and only if $D$ is a tournament.

\begin{Thm}\cite{Ahn1999m}\label{thm:tournament}
 Let $D$ be a tournament with $n$ vertices and $(s_1,s_2,\ldots,s_n)$ be its score sequence.
 Then $C^m(D)$ is as follows:
 \begin{itemize}
   \item[(i)] $C^2(D)=\begin{cases}
                       K_{n-2} \cup I_{2}, & \mbox{if } s_1=0,s_2=1 ;  \\
                       K_{n-1} \cup I_{1}, & \mbox{if } s_1=0,s_2\ge2;
                       \\
                       \mbox{$K_{n}$ or $K_{n}-P_2$ or $K_{n}-P_3$}, & \mbox{if } s_1=1,s_2\ge2;
                       \\
                       \mbox{$K_{n}-P_3$ or $K_{n}-P_4$}, & \mbox{if } s_1=s_2=1,s_3\ge2;
                       \\
                       K_{n}, & \mbox{if } s_1\ge2;
                     \end{cases}$
   \item[(ii)] $C^3(D)=\begin{cases}
                       K_{n}, & \mbox{if either $s_1=1,s_2\ge2$ or $s_1\ge2$};  \\
                       K_{n}-P_2, & \mbox{if } s_1=s_2=1,s_3\ge2;
                       \\
                       K_n-C_3, & \mbox{if } s_1=s_2=s_3=1;
                     \end{cases}$;
   \item[(iii)] for $m\ge 4$, $C^m(D)=\begin{cases}
                       K_n& \mbox{if $s_1=1,s_2\ge2$ or $s_1=s_2=1,s_3\ge2$ or $s_1\ge2$};  \\
                       K_n-C_3, & \mbox{if } s_1=s_2=s_3=1.
                     \end{cases}$
 \end{itemize}
\end{Thm}

Theorem~\ref{thm:tournament} left out the characterization of the $m$-step competition graph of a tournament with sinks for $m \ge 3$.
In this section, we take care of it to eventually compute the competition period and competition index of a tournament with sinks.

\begin{Prop}\label{prop:onesink}
  There is at most one sink in any tournament.
\end{Prop}
\begin{proof}
  Suppose to the contrary that there is a tournament $D$ with at least two sinks.
Let $x$ and $y$ be sinks of $D$.
Since $D$ is a tournament, one of $(x,y)$ or $(y,x)$ must be in $A(D)$, which is a contradiction.
Hence there is at most one sink in any tournament.
\end{proof}

\begin{Cor}
 Let $n$ be an integer with $n\ge3$. Then a tournament of order $n$ is acyclic if and only if its sink elimination index is $n-1$.
\end{Cor}

\begin{Thm}\label{thm:tournament2}
  Let $D$ be a tournament of order $n \ge 2$ with a sink and let $\zeta(D)$ be the sink elimination index of $D$.
  Then, for a positive integer $m$, the following are true:
  \begin{itemize}
    \item[(i)] $1 \le \zeta(D) \le n-1$ with $\zeta(D) \neq n-2$.
        Moreover, for each integer $i$ satisfying $1 \le i \le n-1$ and $i \neq n-2$, there exists a tournament with the sink elimination index $i$;
    \item[(ii)] if $1 \le m < \zeta(D)$, then $C^m(D)$ is the union of the complete graph with vertex set $V(D_{m-1}) \setminus W_{m-1}$ and the empty graph with the vertex set $\bigcup_{i=0}^{m-1}{W_i}$;
  \item[(iii)] if $m \ge \zeta(D)$, then $C^m(D)$ is the union of the complete graph with vertex set $V(D_{\zeta(D)})$ and the empty graph with the vertex set $\bigcup_{i=0}^{\zeta(D)-1}{W_i}$;
  %=\begin{cases}
%                       K_{n-m} \cup I_{m}, & \mbox{if } 1 \le m < \zeta(D) ;  \\
%                       K_{n-\zeta(D)} \cup I_{\zeta(D)}, & \mbox{if } m \ge \zeta(D);
%                     \end{cases}$
    \item[(iv)] $\mathrm{cperiod}(D)=1$;
    \item[(v)] $\mathrm{cindex}(D)=\zeta(D)$.
  \end{itemize}
\end{Thm}
\begin{proof}
  Let $(W_0,\ldots,W_{\zeta(D)})$ be the sink sequence of $D$.
  By the hypothesis that $D$ has a sink, $W_0 \neq \emptyset$.
  Then $|W_0| =1$ by Proposition~\ref{prop:onesink}.
  Since $n \ge 2$, $W_0 \neq V(D)$.
  Thus $\zeta(D) \neq 0$ and so $1 \le \zeta(D) \le n-1$.
  Now $|W_i| =1$ for each integer $0 \le i \le \zeta(D)-1$ by Proposition~\ref{prop:onesink}.

  Suppose $\zeta(D)=n-2$.
  Then $D_{\zeta(D)}$ is a tournament of order $2$, which has a sink and a non-sink, and we reach a contradiction.
  Therefore $\zeta(D) \neq n-2$.

   To show the `moreover' part of the statement (i), fix $i$ such that $1 \le i \le n-1$ and $i \neq n-2$.
   We consider a tournament $D'$ defined by
 $V(D')=\{v_1,v_2,\ldots,v_n\}$ and $A(D')=\{(v_l,v_k) \mid 1 \le k < l \le n\} \setminus \{(v_n,v_{i+1})\} \cup \{(v_{i+1},v_{n})\}$.
 It is easy to check that $W_{j-1}=\{v_{j}\}$ for each $j = 1, \ldots,  i$ and so $\zeta(D') \ge i$.
 On the other hand, there exists a directed cycle $v_{n}\to  v_{n-1} \to v_{n-2} \cdots \to v_{i+1} \to v_n$ in $D'$ and so we may conclude that $\zeta(D') = i$.
Thus the statement (i) is true.

To show the statements (ii) and (iii), we denote the vertices of $D$ as  $v_1,v_2,\ldots,v_n$ so that $W_{i} = \{v_{i+1}\}$ for each integer $0 \le i \le \zeta(D)-1$.
 Then, for each integer $1 \le i \le \zeta(D)$, the length of a longest directed walk with an initial vertex $v_i$ is $i-1$ by Lemma~\ref{lem:walk-length}.
Take a positive integer $m$.
 Suppose $1 \le m < \zeta(D)$.
  Then, for every vertex $v$ in $V(D_{m-1}) \setminus W_{m-1}$, there is an arc from $v$ to $v_m$.
  Concatenating this arc with the directed walk of length $m-1$ from $v_m$ to $v_1$, every vertex in $V(D_{m-1}) \setminus W_{m-1}$ has $v_1$ as an $m$-step prey.
   Since $V(D_{m-1}) \setminus W_{m-1} = V(D) \setminus \bigcup_{i=0}^{m-1}{W_i}$, $V(D_{m-1}) \setminus W_{m-1}$ forms a clique of size $|V(D) \setminus \bigcup_{i=0}^{m-1}{W_i}| = n-m$ in $C^m(D)$.
   Since no vertex in $\bigcup_{i=0}^{m-1}{W_i}$ has an $m$-step prey in $D$ by Lemma~\ref{lem:walk-length}, the vertices in $\bigcup_{i=0}^{m-1}{W_i}$ are isolated in $C^m(D)$.
   Thus the statement (ii) is true.

   Suppose $m \ge \zeta(D)$.
   By (i), $1 \le \zeta(D) \le n-1$ and $\zeta(D) \neq n-2$.
 If $\zeta(D)=n-1$, then there is at most one vertex which has an $m$-step prey in $D$ and so $C^m(D)=I_{n}=K_{1} \cup I_{n-1}$.
 Suppose $1 \le \zeta(D) \le n-3$.
    Then $D_{\zeta(D)}$ is a tournament of order at least $3$ and so $W_{\zeta(D)} \neq V(D_{\zeta(D)})$ by Proposition~\ref{prop:onesink}.
  By the definition of $\zeta(D)$, $W_{\zeta(D)}=\emptyset$, that is, $D_{\zeta(D)}$ is a tournament without sinks.
  Thus, for each vertex $v \in D_{\zeta(D)}$, there exists a directed walk $X_v$ in $D_{\zeta(D)}$ of length $k$ from $v$ to a vertex in $D_{\zeta(D)}$ where $k=m-\zeta(D)$.
  Since $D_{\zeta(D)}$ is a subdigraph of $D$, $X_v$ is a directed walk of length $k$ in $D$.
  By the definition of sink sequence, there is a directed walk of length $\zeta(D)$ from each vertex in  $D_{\zeta(D)}$ to $v_1$ in $D$.
  By concatenating those two directed walks, we obtain a $(v,v_1)$-directed walk of length $m$ for each vertex $v \in D_{\zeta(D)}$ in $D$.
  Hence the vertices in $D_{\zeta(D)}$ forms a clique of size $n-\zeta(D)$ in $C^m(D)$.
  Since no vertex in $\bigcup_{i=0}^{\zeta(D)-1}{W_i}$ has an $m$-step prey in $D$ by Lemma~\ref{lem:walk-length}, the vertices in $\bigcup_{i=0}^{\zeta(D)-1}{W_i}$ are isolated in $C^m(D)$.
  Therefore the statement (iii) is valid.
  Thus the statements (iv) and (v) immediately follow from (ii) and (iii).
\end{proof}

\section{Acknowledgement}
This research was supported by the National Research Foundation of Korea(NRF)  (NRF-2017R1E1A1A03070489 and 2016R1A5A1008055) funded by the Korea government(MSIP).

%\bibliographystyle{plain}
%\bibliography{ref.bib}

\end{document}